\documentclass[12pt,a4paper]{amsart}
\usepackage{amssymb,enumerate}
\usepackage{hyperref}
\usepackage{color}
\usepackage[dvipsnames]{xcolor}
\usepackage[shortlabels]{enumitem}

\newcommand{\NN}{{\mathbb{N}}}
\newcommand{\QQ}{{\mathbb{Q}}}

\newcommand{\bG}{{\mathbf{G}}}
\newcommand{\bZ}{{\mathbf{Z}}}

\newcommand{\cE}{{\mathcal{E}}}
\newcommand{\cG}{{\mathcal{G}}}
\newcommand{\cK}{{\mathcal{K}}}

\newcommand{\fA}{{\mathfrak{A}}}
\newcommand{\fS}{{\mathfrak{S}}}

\newcommand{\sfC}{{\sf{C}}}

\newcommand{\ab}{{\operatorname{ab}}}
\newcommand{\Aut}{\operatorname{Aut}}
\newcommand{\Gal}{\operatorname{Gal}}
\newcommand{\Irr}{\operatorname{Irr}}
\newcommand{\SC}{{\operatorname{sc}}}
\newcommand{\Syl}{\operatorname{Syl}}
\newcommand{\PGL}{\operatorname{PGL}}
\newcommand{\SL}{\operatorname{SL}}

\let\al=\alpha
\let\be=\beta
\let\eps=\epsilon
\let\la=\lambda
\let\vhi=\varphi
\let\si=\sigma
\let\sbs=\subseteq
\let\normal=\unlhd

\def\cent#1#2{{\bf C}_{#1}(#2)}
\def\norm#1#2{{\bf N}_{#1}(#2)}
\def\oh#1#2{{\bf O}_{#1}(#2)}

\theoremstyle{plain}

\newtheorem{thm}{Theorem}[section]
\newtheorem{lem}[thm]{Lemma}
\newtheorem{prop}[thm]{Proposition}
\newtheorem{conj}[thm]{Conjecture}
\newtheorem{cor}[thm]{Corollary}

\newtheorem*{thmA}{Theorem A}

\newtheorem*{conjB}{Conjecture B}

\theoremstyle{definition}
\newtheorem{rem}[thm]{Remark}

\begin{document}

\title[Continuity of $p$-rationality]{The continuity of $p$-rationality of characters\\ and the principal block}

\author{Gunter Malle}
\address[G. Malle]{FB Mathematik, RPTU Kaiserslautern,
  67653 Kaisers\-lautern, Germany.}
\makeatletter
\email{malle@mathematik.uni-kl.de}
\makeatother

\author{J. Miquel Mart\'inez}
\address[J. M. Mart\'inez]{Departament de Matem\`atiques, Universitat
  de Val\`encia, 46100 Burjassot, Val\`encia.}
\makeatletter
\email{josep.m.martinez@uv.es}
\makeatother

\author{Carolina Vallejo}
\address[C. Vallejo]{Dipartamento di Matematica e Informatica `Ulisse Dini',
  50134 Firenze, Italy}
\makeatletter
\email{carolina.vallejorodriguez@unifi.it}
\makeatother

\begin{abstract}
We study rationality properties of irreducible characters of finite groups.
We show that the continuity of $2$-rationality is a phenomenon that can be
detected in the principal $2$-block, thus refining a recent result of
N.\ N.\ Hung. We also propose a conjectural group theoretical criterion for the
continuity gap at level~$1$ for all primes. 
\end{abstract}

\thanks{The first author gratefully acknowledges support by the DFG
 -- Project-ID 286237555 -- TRR 195. The work of the second author is
 funded by the European Union -- Next Generation EU, Missione 4 Componente 1,
 PRIN 2022-2022PSTWLB -- Group Theory and Applications, CUP B53D23009410006 as
 well as the Spanish Ministerio de Ciencia e Innovaci\'on (Grant
 PID2022-137612NB-I00 funded by MCIN/AEI/ 10.13039/501100011033 and “ERDF A way
 of making Europe”). The third author is supported by Rita
 Levi-Montalcini programme (bando 2019) of the Italian Ministero
 dell'Universit\`a e la Ricerca.  She also acknowledges the support of the
 Istituto Nazionale di Alta Matematica (INdAM), as part of the GNSAGA.
 The authors wish to thank Gabriel Navarro, Noelia Rizo and Mandi Schaeffer Fry
 for their comments on an earlier version of this manuscript}

\keywords{$p$-rationality of characters, principal blocks, character
  conductors, character degrees}

\subjclass[2020]{20C15, 20C33}

\dedicatory{Dedicated to Gabriel Navarro on his 60th birthday}

\date{\today}

\maketitle


\section{Introduction}

For $K$ an abelian number field, its \emph{conductor $c(K)$} is
defined as the smallest positive integer $n$ such that $K\sbs\QQ_n$, where
$\QQ_n$ denotes the extension of $\QQ$ generated by a primitive $n$th root of
unity. Given a finite group $G$ and a complex irreducible character
$\chi\in\Irr(G)$, the \emph{conductor of~$\chi$}, denoted $c(\chi)$, is defined
as $c(\QQ(\chi))$, the conductor of the field generated by the values of~$\chi$.
For a prime $p$, the \emph{$p$-rationality level} of $\chi\in\Irr(G)$ is
the non-negative integer $\al$ such that $c(\chi)_p=p^\al$. (Here, we write
$n_p$ for the exact power of $p$ dividing an integer~$n$.) Characters with
specific $p$-rationality level have been extensively studied in the literature.
For example, the $p$-rational characters are the characters with $p$-rationality
level~$0$ and the \emph{almost $p$-rational characters} are the ones with
$p$-rationality level $\leq 1$. Both of these families have been shown to
influence the (local) structure of $G$ (see, for instance
\cite{Isa-Nav-San12, Nav-Riz-Sch-Val21, Hun-Mal-Mar22, Mar-Mar-Sch-Val24}).

In \cite{Hun24}, Hung conjectured that if a finite group $G$ has some
irreducible character~$\chi$ of degree coprime to $p$ with $p$-rationality level
$\al\geq 2$, then $\Irr(G)$ contains characters of degree coprime to $p$ with
$p$-rationality level $\beta$ for any $2\leq\beta\leq \al$. We will
refer to this phenomenon as the \emph{continuity of $p$-rationality}.
In the same paper, he has proven the above conjecture for $p=2$. In the final
section of \cite{Hun24}, Hung suggests the existence of a principal
block version of \cite[Conj.~1.1]{Hun24}:

\begin{conj}[Hung]   \label{conj:Hung}
 Let $G$ be a finite group and $p$ be a prime. If the principal $p$-block
 $B_0(G)$ of $G$ contains an irreducible character of degree coprime to $p$
 with $c(\chi)_p=p^\al$ then it contains irreducible characters of degree
 coprime to $p$ and $c(\psi)_p=p^\beta$ for all $2\leq\beta\leq\al$.
\end{conj}

In our first main result, we confirm Hung's prediction for the prime $p=2$. 

\begin{thmA}
 Conjecture \ref{conj:Hung} holds for $p=2$.
\end{thmA}

Hung's Conjecture is implied by the celebrated Alperin--McKay--Navarro
conjecture \cite[Conj.~B]{Nav04}. In particular, our Theorem~A provides
further evidence for the validity of the Alperin--McKay--Navarro conjecture (at
least for principal blocks and the prime $p=2$), one of the few global-local
conjectures that remain unreduced at the time of this writing.

\smallskip

Assuming \cite[Conj.~A]{Nav-Tie19}, if $\al\geq 2$ is the largest
$p$-rationality level of characters of degree coprime to $p$ of a finite group
$G$ then it is attained at an element of $\Irr(B_0(G))$. Given that
\cite[Conj.~A]{Nav-Tie19} has been proven for $p=2$ in \cite{Mal19}, Theorem~A
above refines the main result of \cite{Hun24}.  Furthermore, using the fact
that \cite[Conj.~A]{Nav-Tie19} has been reduced to a problem on simple groups,
we are also able to reduce Conjecture~\ref{conj:Hung} to a problem on simple
groups (see Theorem~\ref{thm:reduction}) for all primes~$p$.

Interestingly, the continuity of $p$-rationality stops at level~$2$. This is
because there are groups that have an irreducible character of degree coprime to
$p$ and $p$-rationality level~$2$ but none of $p$-rationality level~$1$. That
is to say, in these cases all the irreducible characters of degree coprime to
$p$ that are almost $p$-rational are in fact $p$-rational. An example was
already given in \cite[Rem.~1.6]{Hun24}. We offer here an explanation of this
gap by conjecturally characterizing the finite groups all of whose almost
$p$-rational characters of degree coprime to $p$ are $p$-rational. 

\begin{conjB}   \label{conj:B}
 Let $G$ be a finite group, $p$ be a prime and $P\in \Syl_p(G)$. Write
 $K=\norm G P/\Phi(P)$ and $Q=P/\Phi(P)$.
 Then the following are equivalent:
 \begin{enumerate}[\rm(i)]
  \item No irreducible character of $G$ of degree coprime to $p$ has
   $p$-rationality level $1$.
  \item For every $1\neq y\in Q$ we have:
  \begin{enumerate}[\rm(a)]
   \item $\norm K{\langle y\rangle}/\cent K{\langle y\rangle}\cong\Aut(\langle y\rangle)$, and
   \item $\norm K {\langle y\rangle}/Q$ acts trivially on the set of conjugacy
    classes of $\cent K{y}/Q$.
  \end{enumerate}
 \end{enumerate}
\end{conjB}

Notice that for $p=2$ all three conditions above are trivially satisfied as,
by definition, the conductor of a character cannot have $2$-part equal to $2$.
In Theorem \ref{thm:normaldefect}, we prove that Conjecture~B holds for groups
with a normal Sylow $p$-subgroup. As we will explain in Sections~\ref{sec:2}
and~\ref{sec:normaldefect}, Theorem~\ref{thm:normaldefect} implies that
Conjecture~B would be a consequence of the McKay--Navarro conjecture
\cite[Conj.~A]{Nav04}. In particular, we derive that Conjecture~B holds for
$p$-solvable, sporadic, symmetric and alternating groups, and simple groups of
Lie type in defining characteristic by \cite{Tur13,Nav04,Bru-Nat21,Ruh21}.

It is worth mentioning that Hung speculates that the gap on the continuity of
$p$-rationality might depend on the action of $\norm G P$ on $P/P'$ but does not
know \emph{the precise connection} \cite[Rem.~1.6]{Hun24}. However, both local
conditions in Conjecture~B above are necessary to explain the gap, since the
groups $\sfC_3$ and $\texttt{SmallGroup(216,128)}\cong
(\sfC_3\times\sfC_3 \times\sfC_3)\rtimes{\sf D}_8$ possess
irreducible characters of degree coprime to $3$ and $3$-rationality level 1.

Conjecture B can be seen as a new way of relating the local structure of a group
with its character values. Unlike the phenomena observed in
\cite[Conj.~A]{Nav-Tie19} and \cite[Conj.~1.1]{Hun24}, we note that the absence
of irreducible characters of degree coprime to $p$ of $p$-rationality level~$1$
in a group $G$ is not a property that can be detected in the principal block
(see the end of Section~\ref{sec:normaldefect}). We also offer a version for
principal blocks of Conjecture~B in Conjecture \ref{conj:ppalblock}.
\medskip

Our paper is organized as follows: In Section \ref{sec:2}, we collect some
preliminary results on the $p$-part of character conductors. In
Section~\ref{sec:reduction}, we reduce Conjecture~\ref{conj:Hung} to a problem
on finite simple groups. This problem is settled for the prime~$2$ in
Section~\ref{sec:simple}, thus completing the proof of Theorem A.
Section~\ref{sec:normaldefect} contains the proof of Conjecture B (and its
principal block version) for groups with a normal Sylow $p$-subgroup.

\section{On the $p$-part of character conductors}   \label{sec:2}
For $\al\in\NN$ we denote by $\si_\al\in\Gal(\QQ^\ab/\QQ)$ the Galois
automorphism that fixes $p$-power order roots of unity and sends $p'$-roots of
unity to their $1+p^{\al}$th power. By a slight abuse of notation, we
still denote by $\si_\al$ its restriction to the cyclotomic field $\QQ_n$ for
any positive integer $n$. It is not difficult to see that
$\langle\si_\al\rangle\leq \Gal(\QQ_n/\QQ)$ is a finite $p$-group for any $\al$
and any $n$. Moreover, if $\chi$ is not $p$-rational, then the $p$-rationality
level of $\chi$ is precisely the smallest $\al\geq 1$ such that $\chi$ is
$\si_\al$-invariant (see \cite[Lemma 4.1]{Nav-Tie21}, for instance). While for
$p=2$ the odd-degree 2-rational characters of~$G$ are exactly those fixed by
$\si_1$ by \cite[Thm~2.5]{Val23}, for odd primes we cannot characterise the
$p$-rational characters in terms of the $\si_\al$ simply because
$\Gal(\QQ_{mp^a}/\QQ_m)$ is not a $p$-group.

The celebrated McKay--Navarro conjecture \cite[Conj.~A]{Nav04} implies that
there should exist McKay bijections that preserve the $p$-part of character
conductors. The following is a consequence (of the maximal defect case) of
\cite[Conj.~B]{Nav04}. We state it in the weakest form that suffices for our
purposes. Here, we denote by $\Irr_{p'}(G)$ the set of irreducible characters
of~$G$ whose degree is not divisible by $p$ and by $\Irr(B_0(G))$ the set of
irreducible characters that belong to the principal $p$-block of $G$.

\begin{conj}[Navarro]   \label{conj:navarro}
 Let $G$ be a finite group, $p$ be a prime and $P\in\Syl_p(G)$. There is a
 bijection
 $$^*:\Irr_{p'}(G)\rightarrow\Irr_{p'}(\norm G P)$$
 such that for all $\chi\in\Irr_{p'}(G)$ we have
 \begin{enumerate}[\rm(1)]
    \item $c(\chi)_p=c(\chi^*)_p$, and
    \item $\chi\in\Irr(B_0(G))$ if and only if $\chi^*\in\Irr(B_0(\norm G P))$.
 \end{enumerate}
\end{conj}

It is not difficult to see that Conjecture~\ref{conj:Hung} is a consequence of
Conjecture~\ref{conj:navarro} (and hence of \cite[Conj.~B]{Nav04}). We do so in
Lemma \ref{lem:ConA-Nav} below. We will prove in Section~\ref{sec:normaldefect}
that Conjecture~B (as well as its principal block version) also follows from
Conjecture~\ref{conj:navarro}; however, this will require significantly more
effort.
\medskip

We next collect some useful results on the $p$-part of character conductors.
Throughout, $G$ denotes a finite group.

\begin{lem}   \label{lem:p' index}
 Let $P\in\Syl_p(G)$, $H\leq G$ with $p\nmid |G:H|$ and $\cent G P\sbs H$,
 and let $\al\geq 1$. If $\psi\in\Irr_{0}(B_0(H))$ is fixed by $\si_\al$ then
 there is a $\si_\al$-invariant constituent of $\psi^G$ in $\Irr_{p'}(B_0(G))$.
 In other words, if $c(\psi)_p\ge p$ then there is a constituent $\chi$ of
 $\psi^G$ with $c(\chi)_p\leq c(\psi)_p$.
\end{lem}

\begin{proof}
Mimic the proof of \cite[Lemma 3.7]{Mar-Mar-Sch-Val24} for $\si_\al$ instead of
$\si$, using that $\si_\al$ also has $p$-power order.
\end{proof}

\begin{lem}  \label{lem:constituents have level 2 or more}
 Let $N\normal G$ and $P\in\Syl_p(G)$. Let $\theta\in\Irr_{p'}(B_0(N))$ be
 $P$-invariant and with $c(\theta)_p\geq p^2$, and let
 $\chi\in\Irr_{p'}(B_0(G))$ lie over $\theta$.
 Then $c(\chi)_p\geq p^2$.
\end{lem}

\begin{proof}
Assume that $c(\chi)_p\leq p$ so that $\chi$ is fixed by $\si_1$. Then $\chi_N$
is fixed by $\si_1$. This implies that $\si_1$ permutes the set of
$G$-conjugates of $\theta$, which has size not divisible by~$p$. In particular,
$\si_1$ fixes some $G$-conjugate of $\theta$ so it fixes $\theta$, and
$c(\theta)_p\leq p$, a contradiction. 
\end{proof}

\begin{thm}[Alperin--Dade]   \label{thm:Alperin-Dade}
 Let $N\normal G$ with $|G:N|$ not divisible by $p$. Assume that $G=N\cent G P$
 for $P\in\Syl_p(G)$. Then restriction defines a bijection
 $\Irr(B_0(G))\to\Irr(B_0(N))$. In particular, $c(\chi)_p=c(\chi_N)_p$ for
 every $\chi\in\Irr(B_0(G))$.
\end{thm}

\begin{proof}
The first part was proved by Alperin \cite{Alp76} when $G/N$ is solvable, and by
Dade \cite{Dad77} without the solvability assumption. Since restriction commutes
with Galois action, it follows that $\chi$ is fixed by $\si_\al$ if and only
if $\chi_N$ is fixed by $\si_\al$ for any $\al$.
\end{proof}

\begin{lem}   \label{lem:Hung 2.4}
 Let $N\normal G$ with $|G:N|$ not divisible by $p$. Let $\theta\in\Irr(N)$,
 $\chi\in\Irr(G|\theta)$ and assume $c(\chi)_p\geq p$. Then
 $c(\theta)_p\leq c(\chi)_p$. If, moreover, $c(\chi)_p$ or $c(\theta)_p$ are
 $\geq p^2$, then $c(\chi)_p=c(\theta)_p$.
\end{lem}

\begin{proof}
See \cite[Lemma 2.4(i) and (iii)]{Hun24}.
\end{proof}

\begin{lem}   \label{lem:Hung 3.3}
 Let $\vhi,\psi\in\Irr(G)$. Assume $\chi:=\vhi\psi\in\Irr(G)$ and
 $c(\vhi)_p\leq c(\chi)_p$. Assume further that $G/\ker\psi$ is a $p$-group.
 Then
 \begin{enumerate}[\rm(a)]
  \item $c(\psi)_p\leq c(\chi)_p$, and
  \item either $c(\vhi)_p = c(\chi)_p$ or $c(\psi)_p=c(\chi)_p$.
 \end{enumerate}
\end{lem}

\begin{proof}
See \cite[Lemma 3.3]{Hun24}.
\end{proof}

If $G$ is a $p$-group and $\la\in\Irr_{p'}(G)$, then $G/\ker\la$ is cyclic and
$c(\la)_p=|G/\ker \la|$. In particular, Conjecture~\ref{conj:Hung} trivially
holds for $p$-groups as these have a unique $p$-block, namely the principal
one. We prove below that Conjecture~\ref{conj:Hung} follows from
Conjecture~\ref{conj:navarro}.

\begin{lem}   \label{lem:ConA-Nav}
 Suppose that $G$ satisfies the conclusion of Conjecture \ref{conj:navarro} for
 $p$. Then if the principal $p$-block $B_0(G)$ of $G$ contains irreducible
 characters of degree coprime to $p$ with $c(\chi)_p=p^\al$ then it contains
 irreducible characters of degree coprime to $p$ and $c(\psi)_p=p^\beta$ for
 all $2\leq\beta\leq\al$.
\end{lem}

\begin{proof}
Let $P\in \Syl_p (G)$. By hypothesis, we can assume that $P\normal G$.
Let $\chi \in \Irr_{p'}(B_0(G))$ with $c(\chi)_p=p^\alpha$, and let
$\la \in \Irr( P)$ be under $\chi$. Since $\alpha\geq 2$,
by Lemma \ref{lem:Hung 2.4}, we have that $c(\la)=c(\chi)_p$. Now, $P$ has
a linear character $\la_\beta$ of conductor $p^\beta$ for every
$2\leq \beta \leq \alpha$, as noticed just before the statement of this lemma.
We can take $\chi_\beta \in \Irr(B_0(G))$ over $\la_\beta$ by
\cite[Thm~9.4]{Nav98}. Then $\chi_\beta$ has degree coprime to $p$ and
$c(\chi_\beta)=p^\beta$ using again Lemma~\ref{lem:Hung 2.4}. 
\end{proof}

\section{A reduction theorem for Hung's conjecture}   \label{sec:reduction}

The purpose of this section is to show that if certain conditions hold for
almost/quasi-simple groups and a prime $p$, then Conjecture \ref{conj:Hung}
holds for all finite groups with respect to the prime $p$. 
More specifically, our reduction theorem depends on the following conjecture
concerning the representation theory of the finite non-abelian simple groups.

\begin{conj}   \label{conj:simple}
 Let $p$ be a prime and let $S$ be a finite non-abelian simple group of order
 divisible by $p$. Then the following hold:
 \begin{enumerate}[\rm(a)]
  \item Let $X$ be a quasi-simple group with $X/\bZ(X)\cong S$ and $\bZ(X)$ of
   order coprime to $p$. Let $Y$ be almost quasi-simple with component
   $F^*(Y)=X$ and $Y/X$ a $p$-group. If $Q \in \Syl_p(Y)$ and $\exp(Q/Q')=p^e$,
   then every character in $\Irr_{p'}(Y)$ is $\si_e$-invariant
   (or equivalently, it has $p$-rationality level at most $e$).
  \item Let $A$ be an almost simple group with socle $F^*(A)=S$ such that $A/S$
   is a $p$-group. Assume $\Irr_{p'}(B_0(S))$ contains $A$-invariant characters
   with $p$-rationality level~$\al$. Then $\Irr_{p'}(B_0(S))$ contains
   $A$-invariant characters with $p$-rationality level $\beta$ for every
   $2\leq \beta\leq \al$.
 \end{enumerate}
\end{conj} 

We recall that a group $S$ is said to be involved in a group $G$ if there exist
$N,K \leq G$ with $N\normal K$ and $K/N\cong S$.

\begin{rem}   \label{rem:Navarro-Tiep}
 Let $G$ be a finite group, $P\in\Syl_p(G)$ and assume $\exp(P/P')\geq p^2$.
 By work of Navarro and Tiep in \cite[Thms~B and~5.7]{Nav-Tie19}, if
 Conjecture~\ref{conj:simple}(a) holds for every finite simple group involved in
 $G$, then $\exp(P/P')$ is the largest $p$-rationality level in $\Irr_{p'}(G)$,
 and such largest $p$-rationality level is attained at $\Irr(B_0(G))$. Note
 that Conjecture~\ref{conj:simple}(a) easily implies
 \cite[Conj.~5.4]{Nav-Tie19}. In particular, this has been shown to hold for
 $p=2$ by the main result of \cite{Mal19}.
 \end{rem}

\begin{thm}   \label{thm:reduction}
 Let $G$ be a finite group, $p$ a prime. Assume that
 Conjecture~\ref{conj:simple} holds for every non-abelian simple group
 involved in $G$. If $\Irr_{p'}(B_0(G))$ contains a character of
 $p$-rationality level $\al\geq 2$, then $\Irr_{p'}(B_0(G))$ contains
 characters of $p$-rationality level $\be$ for every $2\leq \be\leq \al$.
\end{thm}

\begin{proof}
We argue by induction on $|G|$. We may assume $\alpha$ is the largest
$p$-rationality level among the $p'$-degree characters in the principal block.
The rest of the proof is divided into steps.
\medskip

\noindent
\textit{Step 1: We may assume that any minimal normal subgroup of $G$ is
semisimple of order divisible by $p$.}

We may assume that $\oh{p'}G=1$ by induction since
$\Irr(B_0(G/\oh{p'}{G}))\sbs\Irr(B_0(G))$. Let $P\in\Syl_p(G)$ and let $N$ be a
minimal normal subgroup of $G$. Assume that $N$ is a $p$-group, so $N\leq P$.
By Remark~\ref{rem:Navarro-Tiep} and since we are assuming
Conjecture~\ref{conj:simple} we have $\exp(P/P')=p^\al$.
Since $\exp(N)=p$ we have
$$\exp(P/P')=\exp((P/N)/(P/N)')\quad\text{or}\quad
  \exp(P/P')=p\cdot\exp((P/N)/(P/N)').$$
In either case, we are done because $\Irr_{p'}(B_0(G/N))$ contains characters
of every $p$-rationality level $2\leq \beta\leq\al-1$ by induction. The claim
follows. 
\medskip

\noindent
\textit{Step 2: Let $N$ be a minimal normal subgroup of $G$. We may assume
 $G/N$ is a $p$-group. In particular, $B_0(G)$ is the unique block covering
 $B_0(N)$.}

Since we are assuming Conjecture \ref{conj:simple}(a), the largest $p$-parts of
the character conductors in $\Irr_{p'}(B_0(G))$ and $\Irr_{p'}(B_0(PN))$
coincide (see Remark~\ref{rem:Navarro-Tiep}). Therefore it suffices to prove
that if $\Irr_{p'}(B_0(PN))$ contains a character of $p$-rationality level
$\beta\geq 2$ then so does $\Irr_{p'}(B_0(G))$. 

For this, let $\psi\in\Irr_{p'}(B_0(PN))$ with $c(\psi)_p=p^\beta$ and write
$H=PN\cent G P$.  By Theorem~\ref{thm:Alperin-Dade}, $\psi$ extends to a
character $\hat\psi\in\Irr_{p'}(B_0(H))$ with $c(\hat\psi)_p=p^\beta$. 
By Lemma~\ref{lem:p' index} there is some $\chi\in\Irr_{p'}(B_0(G))$ lying over
$\hat \psi$ with $c(\chi)_p\leq p^\beta$. 
Notice that $\chi$ lies over $\theta:=\psi_N\in\Irr_{p'}(B_0(N))$. Also notice
that $\psi=\rho \hat \theta$ for some $\rho\in\Irr_{p'}(PN/N)$ by Gallagher's
theorem, where $\hat\theta\in\Irr(PN)$ is the canonical extension of $\theta$
given by \cite[Cor.~6.2 and~6.4]{Nav18}, with $c(\theta)=c(\hat \theta)$.
By Lemma~\ref{lem:Hung 3.3}, we have $c(\rho)_p\leq c(\psi)_p=p^\beta$ and
either (1)~$c(\hat \theta)_p=c(\psi)_p$ or (2)~$c(\rho)_p=c(\psi)_p$. 

In case (1) we have that $c(\theta)_p=p^\beta\geq p^2$. By
Lemma~\ref{lem:constituents have level 2 or more} this yields that
$c(\chi)_p\geq p^2$ and by using Lemma~\ref{lem:Hung 2.4} we conclude
$c(\chi)_p=p^\beta$, as desired.

In case (2), we have that $c(\rho)_p=p^\beta$, so $\exp(PN/(PN)')\geq p^\beta$
because $PN/\ker\rho$ is cyclic and $(PN)'=P'N\sbs \ker \rho$ (recall that $N$
is perfect by Step~1).
If all characters in $\Irr_{p'}(B_0(G/N))$ have $p$-rationality level strictly
smaller than $\beta$ then Remark~\ref{rem:Navarro-Tiep} implies
$\exp(PN/(PN)') < p^\beta$, which is a contradiction. Hence some character in
$\Irr_{p'}(B_0(G/N))$ has $p$-rationality level at least $\beta$
and the inductive hypothesis guarantees the existence of a character in
$\Irr_{p'}(B_0(G/N))\sbs\Irr_{p'}(B_0(G))$ of $p$-rationality level
exactly~$\beta$.

The last claim follows from \cite[Cor.~9.6]{Nav98}.
\medskip

\noindent
\textit{Final step:}
Let $\chi\in\Irr_{p'}(B_0(G))$ with $c(\chi)_p=p^\al\geq p^2$. Let $N$ be a
minimal normal subgroup of $G$. By Step 1, we can write
$N=S_1\times\dots \times S_t$ where the $S_i$ are isomorphic non-abelian
simple groups of order divisible by $p$ permuted transitively by~$G$. By
Step~2, $G/N$ is a $p$-group. In particular, 
$\chi_N=\theta\in\Irr_{p'}(B_0(N))$ is irreducible with
$c(\theta)_p=p^\gamma\leq p^\al$.

Arguing as in the second paragraph of Step 2, we can write
$\chi=\hat\theta\rho$ where $\hat\theta$ is the canonical extension of $\theta$
to $G$ and $\rho\in\Irr(G/N)$ with $p^\al=c(\rho)_p$ or $p^\al=c(\theta)_p$.
If $p^\al=c(\rho)_p$ then, using that the result is true for $p$-groups
we are done. 
Otherwise, $p^\al=c(\theta)_p$. Let $\{x_1,\ldots,x_t\}$ be a transversal of
$\norm G {S_1}$ in $G$, chosen so that $S_i=(S_1)^{x_i}$. Then we can write
$\theta=\theta_1^{x_1}\times\dots\times\theta_t^{x_t}$, where
$\theta_i\in\Irr_{p'}(B_0(S_1))$.  Since $\theta$ is $G$-invariant, $\theta_i$
and $\theta_1$ are $\Aut(S_1)$-conjugate, for every $i\in\{ 1,\ldots,t\}$.
In particular, $c(\theta_i)_p=c(\theta_1)_p$ for every $i\in\{ 1,\ldots,t\}$
and actually $c(\theta_1)=c(\theta)_p=p^\alpha$. 

As $\theta$ is $G$-invariant, $\theta_1$ is $\norm G {S_1}$-invariant. Since we
may
view $\norm G{S_1}/\cent G{S_1}\leq\Aut(S_1)$ and $\norm G{S_1}/S_1\cent G{S_1}$
is a $p$-group, it follows from Conjecture~\ref{conj:simple}(b) that $S_1$
has $\norm G {S_1}$-invariant characters $\eta_{\beta}\in\Irr_{p'}(B_0(S_1))$
with $c(\eta_\beta)_p=p^\beta$ for every $2\leq\beta\leq\al$.
It is straightforward to check that
$$\theta_\beta:=\eta_\beta^{x_1}\times\dots\times\eta_{\beta}^{x_t}$$
is $G$-invariant and has degree coprime to $p$. Further, $\theta_\beta$ lies in
$B_0(N)$ by \cite[Lemma~2.1(b)]{Riz-Sch-Val20}. 
Notice also that $c(\theta_\beta)_p=c(\eta_\beta)_p=p^\beta$. By
\cite[Cor.~6.2 and 6.4]{Nav18} there is a canonical extension
$\hat\theta_\beta\in\Irr(G)$ of $\theta_\beta$ with
$c(\hat\theta_\beta)=c(\theta_\beta)$, so $c(\hat\theta_\beta)_p=p^\beta$.
Since $G/N$ is a $p$-group, using \cite[Cor.~9.6]{Nav98} we conclude that 
$\hat\theta_\beta$ lies in the principal $p$-block, as desired.
\end{proof}

\section{On almost simple groups and Theorem A}\label{sec:simple}

The purpose of this Section is to show that Conjecture \ref{conj:simple}(b)
holds for every finite simple group when $p=2$. Our first observation is for
arbitrary primes~$p$:

\begin{prop}   \label{prop:spor}
 Conjecture \ref{conj:simple}(b) holds for any prime $p$ whenever $S$ is a
 sporadic simple group, the Tits group, an alternating group $\fA_n$ with
 $n\ge5$, or a simple group of Lie type in characteristic~$p$.
\end{prop}

\begin{proof} 
As argued in the proof of \cite[Prop.~5.1]{Hun24}, alternating groups, sporadic
simple groups and groups of Lie type in characteristic~$p$ have
$c(\chi)_p\leq p^2$ for all of their irreducible characters~$\chi$.
\end{proof}

To deal with the groups of Lie type in non-defining characteristic we use the
following setup. Let $\bG$ be a simple algebraic group of simply connected type
over an algebraically closed field of characteristic~$r$ and $F:\bG\to\bG$ a
Steinberg endomorphism, with finite group of fixed points $G=\bG^F$. By
Proposition~\ref{prop:spor}, to show Conjecture~\ref{conj:simple}(b) it remains
to consider the non-abelian simple quotients~$S$
of groups $G$ as just described, for primes $p\ne r$. Assume $S=G/\bZ(G)$. Now
the irreducible characters $\chi\in\Irr_{p'}(B_0(S))$ are the deflations of
characters $\tilde\chi\in\Irr_{p'}(B_0(G))$ containing $\bZ(G)$ in their kernel,
and all automorphisms of $S$ are induced by automorphisms of $G$, so to prove
the validity of Conjecture~\ref{conj:simple}(b) we can argue using characters
of~$G$. We refer to \cite{GM20} for notation and background on characters of
finite reductive groups.

The following observation might be of independent interest:

\begin{prop}   \label{prop:oddrat}
 Let $G$ be as above. Then any unipotent character of $G$ of odd degree is
 rational valued.
\end{prop}

\begin{proof}
By \cite[Prop.~4.5.5]{GM20} the character fields of unipotent characters are
generated by their Frobenius eigenvalues, except for the principal series
characters of $E_7$ and $E_8$ lying in exceptional families. Now for classical
types all Frobenius eigenvalues are $\pm1$, and for the exceptional types a
quick inspection shows that all irrational unipotent characters have even
degree.
\end{proof}

\begin{thm}   \label{thm:p=2}
 Conjecture \ref{conj:simple} holds for $p=2$.
\end{thm}

\begin{proof}
As pointed out in Remark~\ref{rem:Navarro-Tiep}, Part~(a) has been shown in
\cite{Mal19}. For Part~(b), by Proposition~\ref{prop:spor} we only need to
consider $S$ a simple group of Lie type in odd characteristic~$r$. Let $\bG,F$
and $G=\bG^F$ be as above such that $S=G/\bZ(G)$, and let $P\in\Syl_2(G)$. By
\cite[Prop.~3.9]{Nav-Tie19}, $P/P'$ has exponent~2 and thus by
\cite[Thm~1]{Mal19} all $\chi\in\Irr(B_0(S))$ have 2-rationality level $\al\le2$
and there is nothing to prove, unless $G$ is either $E_6(\eps q)_\SC$ or
$\SL_{2n}(\eps q)$ with $q\equiv\eps\pmod4$ and $n$ not a 2-power. It remains to
discuss the latter two families of groups. 
\par
We first make some general observations. By Brou\'e--Michel
\cite[Thm~9.12]{CE04}, the principal $2$-block of $G$ lies in the union of
Lusztig series
$$\Irr(B_0(G))\subseteq\bigsqcup_t\cE(G,t)$$
where $t$ runs over $2$-elements in a group $G^*$ dual to $G$. Furthermore, by
Lusztig's character degree formula (Jordan decomposition,
\cite[Thm~2.6.22]{GM20}), for $\chi\in\cE(G,t)$
to have odd degree, $t$ has to be 2-central in $G^*$, that is to say, it must
lie in the centre of some Sylow $2$-subgroup of~$G^*$. By
Proposition~\ref{prop:oddrat} the claim is vacuous for unipotent characters.
Let $A$ be such that $G\unlhd A$.

Now first assume $G=E_6(\eps q)_\SC$ with $4|(q-\eps)$. By \cite[Tab.~1]{Mal19}
the centre $Z$ of a Sylow 2-subgroup of $G$ is a cyclic group of order
$(q-\eps)_2$ and is centralised by a (Levi) subgroup $D_5(\eps q)(q-\eps)$.
Now note that all non-trivial elements $1\ne t\in Z$ have connected centraliser
(of type $D_5$). Thus, by \cite[Thm~5.6]{Hun24} the $2$-rationality level
$\al_t$ of all $\chi\in\cE(G,t)$ is the same and given by $|t|=2^{\al_t}$.
Now, as $\mathbf{C}_{\bG^*}(t)$ is connected, for each $t$ there is a unique
semisimple character $\chi_t\in\cE(G,t)$ \cite[Def.~2.6.9, Cor.~2.6.18]{GM20}.
By \cite[Thm~B]{En00} all $\chi_t$ lie in the principal
2-block of~$G$. If $\cE(G,t)$ contains an $A$-invariant character, then in
particular $\cE(G,t)$ is $A$-stable, and hence $\chi_t$ is $A$-invariant.
But then also all $\cE(G,t^k)$ are $A$-stable for any $k\ge1$ by
\cite[Prop.~7.2]{Tay18}. Since $\chi_{t^k}$ is the unique semisimple character
in $\cE(G,t^k)$ this forces $\chi_{t^k}$ to be $A$-invariant. Thus any
$2\le\be\le\al_t$ occurs as $2$-rationality level of some
$\chi\in\Irr_{2'}(B_0(S))$, as desired.
\par
Finally, assume $G=\SL_{2n}(\eps q)$ with $4|(q-\eps)$ and $n$ not a 2-power.
In this case, by \cite[Tab.~4.5.1]{GLS}, all centralisers of 2-central
2-elements of $G^*\cong\PGL_{2n}(\eps q)$ are connected, and they are Levi
subgroups. Moreover, by \cite[Thm~21.14]{CE04} all Lusztig series $\cE(G,t)$,
for 2-elements $t\in G^*$ lie in the principal 2-block. But then we can argue
exactly as in the previous case.
\end{proof}

\section{On the $p$-rationality gap}\label{sec:normaldefect}

The purpose of this section is to show that Conjecture B of the introduction
holds for groups with a normal Sylow $p$-subgroup. In order to do so we need to
recall the concept of semi-inertia subgroup and some of its properties. 
 
Let $N \normal G$. The \emph{semi-inertia subgroup of $\theta\in\Irr(N)$} is
defined as
$$G_\theta^*:=\{g \in G\mid
   \theta^g =\theta^\tau\text{ for some }\tau\in \Gal(\QQ(\theta)/\QQ)\}.$$
Notice that $G_\theta\leq G_\theta^*$.

\begin{lem}   \label{lem:semi inertia}
 Let $N\normal G$ and $\theta\in\Irr(N)$. Let $G_{\theta}^*$ be the semi-inertia
 subgroup of $\theta$ in $G$.
 \begin{enumerate}[\rm(a)]
  \item If $\psi\in\Irr(G_{\theta}|\theta)$ then
   $\QQ(\psi^{G_{\theta}^*})=\QQ(\psi^G)$.
  \item The map $g\mapsto \tau$ where $\theta^g=\theta^\tau$ is a well-defined
   group homomorphism
   $$G_{\theta}^*\to\Gal(\QQ(\theta)/\QQ(\theta^G))$$
   with kernel $G_{\theta}$, so
   $$G_{\theta}^*/G_{\theta}\cong \Gal(\QQ(\theta)/\QQ(\theta^G)).$$
 \end{enumerate}
\end{lem}

\begin{proof}
See \cite[Lemmas 2.2 and 2.3]{Nav-Ten10}. Note that if $g\mapsto \tau$, then
$(\theta^G)^\tau=(\theta^\tau)^G=(\theta^g)^G=\theta^G$, so $\tau$ fixes
$\theta^G$, and it follows that the map is well-defined.
\end{proof}

We fix a prime $p$ and set $\cG:=\Gal(\QQ_p/\QQ)=\{\tau_1,\ldots,\tau_{p-1}\}$
where $\tau_j$ takes a
fixed primitive $p$th root of unity to its $j$th power. If $P$ is an elementary
abelian $p$-group, then $\cG$ acts on $P$ by automorphisms as $\tau_j(y)=y^j$ for
every $y\in P$ and $\tau_j\in\cG$. 

\begin{lem}   \label{lem:permutationisomorphism}
 Let $P$ be an elementary abelian $p$-group. Suppose that $X$ acts coprimely
 by automorphisms on $P$. Let $\cG=\Gal(\QQ_p/\QQ)$ and write $A=X\times \cG$.
 Then $A$ acts by automorphisms and coprimely on $P$ (where the action of $\cG$
 is as in the previous paragraph) and the actions of $A$ on $P$ and $\Irr(P)$
 are permutation isomorphic. Set $G=P\rtimes X$. If $y\mapsto\la_y$ is an
 $A$-equivariant bijection $P\to\Irr(P)$, then 
 $$\norm G {\langle y \rangle }/\cent G {y}
   =G_{\la_y}^*/G_{\la_y}\qquad\text{for all $y\in P$} .$$
\end{lem}

\begin{proof}
The group $A$ acts by automorphisms on $P$ because the actions of $X$ and $\cG$
commute. The actions of $A$ on $P$ and $\Irr(P)$ are permutation isomorphic by
\cite[Thm~18.10]{Hup98}.
Let $\al:P\to\Irr(P)$ be an $A$-equivariant bijection. Write $\al(y)=\la$. 
Notice that $G_\la=P\rtimes X_\la$. Now $x\in X_\la$ if, and only if,
$\al(y)^x=\al(y)$ if, and only if, $x\in\cent X y$. Also,
$G_\la^*=P\rtimes X_\la ^*$ and $x\in X_\la^*$ if, and only if, there is some
$\tau\in\cG$ for which $\al(y)^x=\al(y)^\tau$ if, and only if,
$y^x=y^\tau=y^j$ for some $j\in\{ 1,\ldots,p-1\}$ if, and only if,
$x\in\norm X {\langle y\rangle}$.
\end{proof}

We denote by $\Irr_{p',\si_1}(G)$ the subset of characters in $\Irr_{p'}(G)$
fixed by the Galois automorphism $\si_1$, that is, the set of almost
$p$-rational characters of degree coprime to $p$. We mention that there is no
standard notation yet for this set in the literature (see \cite{Riz-Sch-Val20,
Hun-Mal-Mar22}), we use the notation from \cite{Nav-Riz-Sch-Val21}.

\begin{lem}   \label{lem:frattini}
 Let $G$ be a finite group, $P\in\Syl_p(G)$ and assume $P\normal G$. Then
 $\Irr_{p',\si_1}(G)=\Irr(G/\Phi(P))$.
\end{lem}

\begin{proof}
Let $\chi\in\Irr_{p',\si_1}(G)$. Then by Clifford's theorem
$\chi_P=e(\la_1+\dots+\la_t)$ for some $e\ge1$, where $\{\la_1,\dots,\la_t\}$
is the set of $G$-conjugates of some linear $\la\in\Irr(P/P')$. Since
$\chi^{\si_1}=\chi$ the $p$-group $\langle \si_1\rangle$ acts on
$\{\la_1,\dots,\la_t\}$, but $t=|G:G_\la|$ divides $|G:P|$ and thus is prime
to~$p$. This shows $\si_1$ fixes some
$\la_i$ and therefore it fixes all of them. Thus $\la^{\si_1}=\la^{p+1}=\la$ so
$\la$ has order $p$ as an element of $\Irr(P/P')$. We conclude that
$\Phi(P)\sbs\ker\la$ so $\Irr_{p',\si_1}(G)\sbs\Irr(G/\Phi(P))$. For the
reverse inclusion, notice that $\la\in\Irr(P/\Phi(P))$ is $\si_1$-fixed, and so
is every $\chi\in\Irr(G|\la)$ by \cite[Lemma~5.1]{Nav-Tie19}. 
\end{proof}

We recall the statement of Conjecture~B below for the reader's convenience.

\begin{conj}   \label{conj:continuity gap} 
 Let $G$ be a finite group, $p$ be a prime and $P\in\Syl_p(G)$. Write
 $K=\norm G P/\Phi(P)$ and $Q=P/\Phi(P)$. Then the following are equivalent:
 \begin{enumerate}[\rm(i)]
  \item Every almost $p$-rational $\chi \in \Irr_{p'}(G)$ is $p$-rational. 
  \item The following hold for all $1\neq y\in Q$:
  \begin{enumerate}[\rm(a)]
   \item $\norm K {\langle y\rangle}/\cent K{y}\cong \Aut(\langle y\rangle)$,
    and
   \item $\norm K {\langle y \rangle}/Q$ acts trivially on the set of conjugacy
    classes of $\cent K{y}/Q$.
  \end{enumerate}
 \end{enumerate}
\end{conj}

Recall that if $p=2$, the conditions in Conjecture~\ref{conj:continuity gap} are
vacuous. If $p$ is odd, then a character $\chi$ is almost $p$-rational if, and
only if, $\chi$ is $\si_1$-fixed.

\begin{thm}   \label{thm:normaldefect}
 Let $G$ be a finite group, $p$ be a prime and $P\in \Syl_p(G)$. Suppose that
 $P\normal G$. Then Conjecture~\ref{conj:continuity gap} holds for $G$ and $p$.
\end{thm}

\begin{proof}
Conjecture~\ref{conj:continuity gap} trivially holds when $p=2$, so we may
assume that $p$ is odd.
We begin the proof with a few observations. Let $Q=P/\Phi(P)$ and write
$\cG=\Gal(\QQ_p/\QQ)$. If $m=|G|_{p'}$ then restriction yields an isomorphism
$\Gal(\QQ_{mp}/\QQ_m)\rightarrow\cG$ and by abuse of notation, we identify the
elements $\tau\in\Gal(\QQ_{mp}/\QQ_m)$ with their restriction $\tau\in\cG$.
Consider the action by automorphisms of $\cG$ on $Q$ described before
Lemma~\ref{lem:permutationisomorphism}. Write $K=Q\rtimes X$ for some
complement $X$ of $Q$ in $G/\Phi(P)$. By Lemma~\ref{lem:frattini} we have
$\Irr_{p',\si_1}(G)=\Irr(K)$. Write also $A=X \times \cG$. Notice that every
$\la\in\Irr(Q)$ extends to a canonical $\hat\la\in\Irr(K_\la)$ by
\cite[Cor.~6.2 and 6.4]{Nav18}, and $\QQ(\hat\la)=\QQ(\la)=\QQ_p$.
If $H=K^*_\la$ then it follows that $H_{\hat\la}=H_\la=K_\la$ and that
$H^*_{\hat\la}=H$. By combining Clifford's and Gallagher's theorem we have
$$\Irr(K|\la)=\{(\rho\hat\la)^K	\mid\rho\in\Irr(K_\la/Q)\}.$$
Furthermore if $\chi\in\Irr(K|\la)$ has Clifford correspondent
$\psi\in\Irr(K_\la|\la)$ then we denote by
$\psi^*=\psi^{K^*_\la}\in\Irr(K^*_\la|\la)$ so that $\QQ(\psi^*)=\QQ(\chi)$,
where $\chi=(\psi^*)^G$, by Lemma~\ref{lem:semi inertia}.
\medskip

We begin by proving that (i) implies (ii.a), so we are assuming that every
element of $\Irr(K)$ is $p$-rational. Suppose that for some $1\neq y \in Q$ we
have that $\norm K {\langle y \rangle}/\cent{K}{y}$ is not isomorphic to $\cG$;
in other words, that
$$|\norm K {\langle y \rangle}/\cent{K}{y}|<p-1.$$
Let $1_Q\neq \la \in \Irr(Q)$ correspond to $y$ under an $A$-equivariant
bijection as in Lemma~\ref{lem:permutationisomorphism}, so that
$K_\la^*/K_\la =\norm K {\langle y\rangle}/\cent K {y}$. Let $\chi \in \Irr(K)$
be over $\la$. By hypothesis $\chi$ is $p$-rational. Let
$\psi\in\Irr(K_\la | \la)$ be the Clifford correspondent of $\chi$.
As explained in the first paragraph, then $\psi^*=\psi^{K_\la^*}$ is also
$p$-rational. In particular we have 
$$\psi^*(g)=\psi(1)\sum_{k \in [K_\la^*/K_\la]}\la^k(g)\in\QQ
  \qquad\text{for every $g \in Q$},$$
where $[K_\la^*/K_\la]$ is a complete set of representatives of the
$K_\la$-cosets in $K_\la^*$. Let $\cK \leq \cG$ be the image of $K_\la^*/K_\la$
under the homomorphism from Lemma~\ref{lem:semi inertia}, and notice that
$\cK<\cG$. Then $\displaystyle \sum_{\tau \in \cK} \la^\tau$ is rational. But
for $\si\in\cG\setminus\cK$, the equality
$\displaystyle\sum_{\tau\in\cK} \la^\tau= \sum_{\tau\in\cK} \la^{\tau\si}$
contradicts the linear independence of characters of $Q$. This proves that (i)
implies (ii.a).

Now we prove (i) implies (ii.b). Let $1_Q\neq\la\in\Irr(Q)$. Notice that the
previous proof shows $K^*_\la/K_\la\cong\cG$. Now $\Irr(K|\la)$ are all
$p$-rational by hypothesis. Let $\chi\in\Irr(K|\la)$, let
$\psi\in\Irr(K_\la|\la)$ be its Clifford correspondent and let
$\psi^*=\psi^{K^*_\la}$ be its correspondent in $K^*_\la$ and write
$\psi=\rho\hat\la$ as in the first paragraph of this proof. Now, $\psi$ has
values contained in $\QQ_{pm}$ where $m=|G|_{p'}$. Since $\chi$ is $p$-rational,
if $\tau\in\Gal(\QQ_{pm}/\QQ_m)\cong\cG$ we have $\chi^\tau=\chi$ so
$(\psi^*)^\tau=\psi^*$. Further, by the first comment of this paragraph, there
is $x\in K_\la^*/K_\la$ with $\la^\tau=\la^x$. This implies that both
$\psi^\tau$ and $\psi^x$ are the Clifford correspondents of $\psi^*$ in $K_\la$
over $\la^\tau$, and therefore $\psi^\tau=\psi^x$. Using that $\rho^\tau=\rho$,
$$\rho\hat\la^\tau=\rho^\tau\hat\la^\tau=\psi^\tau=\psi^x
  =\rho^x\hat\la^x=\rho^x\hat\la^\tau$$
which implies $\rho=\rho^x$. Since this holds for any $\chi\in\Irr(K_\la|\la)$
and any $\tau\in\cG$, we conclude that $\rho^x=\rho$ for
any $\rho\in\Irr(K_\la/Q)$ and any $x\in K_\la^*$. By Brauer's Lemma on
character tables (see \cite[Thm~2.3]{Nav18}) we have that $K_\la^*/Q$ acts
trivially on the set of conjugacy classes of $K_\la/Q$, and this concludes the
proof of (ii.b) assuming~(i).
\medskip

We finally prove that (ii) implies (i). By Lemma \ref{lem:frattini} we need to
show that every character of $K$ is $p$-rational. Let $\chi \in\Irr(K)$ and let
$\la\in\Irr(Q)$ be under $\chi$. If $\la$ is trivial, then $\chi$
can be seen as a character of the $p'$-group $K/Q$ and
hence $\chi$ is $p$-rational, as wanted. We can therefore assume that $\la$ is
nontrivial. By Lemma~\ref{lem:permutationisomorphism} there is some
$1\neq y\in Q$ such that $K_\la^*/K_\la=\norm K {\langle y\rangle}/\cent K{y}$.
By hypothesis and Lemma~\ref{lem:semi inertia}(b), we have
$K_\la^*/K_\la\cong \cG$. Let $\psi \in \Irr(K_\la | \la)$ be the Clifford
correspondent of $\chi$ and write $\psi^*=\psi^{K_\la^*}$. We have that
$\QQ(\chi)=\QQ(\psi^*)$ by Lemma~\ref{lem:semi inertia}(a), so we want to show
that $\psi^*$ is $p$-rational. 

By definition, for all $x\in K_\la^*$ there is $\tau\in\cG$ with
$\la^\tau=\la^x$. Since $|K_\la^*/K_\la|=|\cG|$, the converse is also true.
So, fixing some $\tau\in\cG$ we have $\la^\tau=\la^x$ for $x\in K^*_\la$. 
Now, using the notation established at the beginning of this proof, we have
$\hat\la^x=\hat\la^\tau$, where $\hat\la$ is the canonical extension of $\la$
to $K_\la$. Thus
$$\psi^\tau=\rho^\tau\hat\la^\tau=\rho\hat\la^x=\rho^x\hat\la^x=\psi^x$$
where we have used in the second to last equality that $\rho^x=\rho$ by (ii.b)
and Brauer's permutation lemma. Thus $\psi^\tau=\psi^x$ and therefore
$$(\psi^*)^\tau=(\psi^\tau)^{K_\la^*}=(\psi^x)^{K_\la^*}=\psi^*$$
and we conclude that $\psi^*$ is $\tau$-invariant for any $\tau\in\cG$. Thus
$\psi^*$ is $p$-rational and we are done.
\end{proof}

\begin{cor}\label{cor:Hung-Navarro}
 Let $G$ be a finite group, $p$ a prime. If $G$ satisfies the McKay--Navarro
 conjecture for the prime $p$ then Conjecture~\ref{conj:continuity gap} holds
 for $G$ and $p$.
\end{cor}

\begin{proof}
Let $P\in\Syl_p(G)$. As explained at the beginning of Section \ref{sec:2}, if
\cite[Conj.~A]{Nav04} holds for $G$ and $p$ then there is a bijection
$^*:\Irr_{p'}(G)\rightarrow\Irr_{p'}(\norm G P)$ with $c(\chi)_p=c(\chi^*)_p$
for all $\chi\in\Irr_{p'}(G)$. It follows that the numbers of $p$-rational and
almost $p$-rational characters in both $G$ and $\norm G P$ coincide. Therefore
we may work in $\norm G P$ and then we can apply Theorem~\ref{thm:normaldefect}.
\end{proof}

As already mentioned in the Introduction, thanks to
Corollary~\ref{cor:Hung-Navarro} and \cite{Bru-Nat21,Nav04,Ruh21,Tur13},
Conjecture~\ref{conj:continuity gap} (Conjecture B from the Introduction)
holds for $G$ and $p$ whenever $G$ is $p$-solvable group, a group of Lie type
in characteristic $p$, a sporadic, a symmetric or and alternating group. 

There is a natural version of Conjecture~\ref{conj:continuity gap} for principal
blocks.

\begin{conj}   \label{conj:ppalblock} 
 Let $G$ be a finite group, $p$ be a prime and $P\in \Syl_p(G)$. Write
 $K=\norm G P/\oh{p'}{\norm G P}\Phi(P)$ and $Q\in\Syl_p(K)$. Then the following
 are equivalent:
 \begin{enumerate}[\rm(i)]
  \item Every almost $p$-rational $\chi \in \Irr_{p'}(B_0(G))$ is $p$-rational. 
  \item The following hold for all $1\neq y\in Q$:
   \begin{enumerate}[\rm(a)]
    \item $\norm K {\langle y\rangle}/\cent K{y}\cong\Aut(\langle y\rangle)$,
     and
    \item $\norm K {\langle y \rangle}/Q$ acts trivially on the set of conjugacy
     classes of $\cent K{y}/Q$.
   \end{enumerate}
 \end{enumerate}
\end{conj}

As explained at the beginning of Section~\ref{sec:2}, the
Alperin--McKay--Navarro Conjecture \cite[Conj.~B]{Nav04} implies
Conjecture~\ref{conj:navarro}.

\begin{cor}   \label{cor:Hung-Navarro-ppalblock}
 Let $G$ be a finite group and $p$ a prime. If $G$ satisfies
 Conjecture~\ref{conj:navarro} for $p$ then Conjecture~\ref{conj:ppalblock}
 holds for $G$ and~$p$.
\end{cor}

\begin{proof}
By Conjecture~\ref{conj:navarro} we may assume $P\normal G$.
By \cite[Thm~10.20]{Nav98}, in this case $\Irr(B_0(G))=\Irr(G/\oh{p'}G)$.
We apply Theorem~\ref{thm:normaldefect} to $G/\oh{p'}G$ to conclude.
\end{proof}

By Corollary \ref{cor:Hung-Navarro-ppalblock} and
\cite{Nav04, Tur13, Gia21, Bru-Nat21}, Conjecture~\ref{conj:ppalblock} holds
for $G$ and $p$ whenever $G$ is $p$-solvable, sporadic, a symmetric or an
alternating group. 
\smallskip

We close by mentioning that the dihedral group $G={\sf D}_{24}$ possesses
irreducible characters of degree
coprime to $3$ and of $3$-rationality level $1$, but none of them lie in the
principal $3$-block. Notice that in fact $G/\oh{3'}G\cong\fS_3$ so the
conclusion of Conjecture~\ref{conj:ppalblock}(ii.b) is satisfied, but not the
one of Conjecture~\ref{conj:continuity gap}(ii.b).


\end{document}